\documentclass[11pt]{amsart}
\usepackage[margin=1.35in]{geometry}
\usepackage{amscd,amsmath,amsxtra,amsthm,amssymb,stmaryrd,xr,mathrsfs,mathtools,enumerate,commath, comment}
\usepackage{stmaryrd}
\usepackage{xcolor}
\usepackage{commath}
\usepackage{comment}
\usepackage{tikz-cd}
\usepackage{longtable} % for 'longtable' environment
\usepackage{pdflscape} % for 'landscape' environment
\usepackage{booktabs}
\usepackage{hyperref}
\hypersetup{
 colorlinks=true,
 linkcolor=blue,
 filecolor=magenta,      
 urlcolor=brown,
 citecolor=orange,
 pdftitle={Arithmetic statistics and diophantine stability},
 pdfpagemode=FullScreen,
    }
\usepackage{enumerate}
\usepackage[utf8]{inputenc}
\usepackage[OT2,T1]{fontenc}
\usepackage{color}

\DeclareSymbolFont{cyrletters}{OT2}{wncyr}{m}{n}
\DeclareMathSymbol{\Sha}{\mathalpha}{cyrletters}{"58}

\newcommand\Q{{\mathbb Q}}

\newcommand\Z{{\mathbb Z}}
\newcommand\F{{\mathbb F}}

\newcommand{\op}[1]{\operatorname{#1}}

\DeclareMathOperator{\Zx}{\mathbb{Z}_p\llbracket x\rrbracket}

\newtheorem*{Theorem*}{Theorem}
\newtheorem{Th}{Theorem}[section]
\newtheorem{Lemma}[Th]{Lemma}
\newtheorem{question}{Question}
\newtheorem*{Ques*}{Question}
\newtheorem{Proposition}[Th]{Proposition}
\newtheorem{Remark}[Th]{Remark}
\newtheorem{Corollary}[Th]{Corollary}
\newtheorem{Conjecture}[Th]{Conjecture}
\newtheorem{Definition}[Th]{Definition}
\newtheorem*{Conj*}{Conjecture}
\newtheorem{lthm}{Theorem}

\theoremstyle{remark}
\newtheorem*{Rem*}{Remark}
\theoremstyle{definition}

\newcommand\mtx[4] { \left( {\begin{array}{cc}
   #1 & #2 \\
   #3 & #4 \\
  \end{array} } \right)}

\begin{document}

\title[Arithmetic statistics and Diophantine Stability]{Arithmetic statistics and Diophantine Stability for Elliptic Curves}

\author[A.~Ray]{Anwesh Ray}
\address[Ray]{Department of Mathematics\\
University of British Columbia\\
Vancouver BC, Canada V6T 1Z2}
\email{anweshray@math.ubc.ca}

%\author[L.~C.~Washington]{Lawrence C. Washington}
%\address[Washington]{Department of Mathematics\\
%University of Maryland}


\begin{abstract}
We study the growth and stability of the Mordell-Weil group and Tate-Shafarevich group of an elliptic curve defined over the rationals, in various cyclic Galois extensions of prime power order. Mazur and Rubin introduced the notion of diophantine stability for the Mordell-Weil group an elliptic curve $E_{/\Q}$ at a given prime $p$. Inspired by their definition of stability for the Mordell-Weil group, we introduce an analogous notion of stability for the Tate-Shafarevich group, called $\Sha$-stability. Using methods in arithmetic statistics and Iwasawa theory, we study the diophantine stability of elliptic curves on average. First, we prove results for a fixed elliptic curve $E$ and varying prime $p$. It is shown that any non-CM elliptic curve of rank 0 defined over the rationals is diophantine stable and $\Sha$-stable at $100\%$ of primes $p$. Next, we show that standard conjectures on rank distribution give lower bounds for the proportion of rational elliptic curves $E$ that are diophantine stable at a fixed prime $p\geq 11$. Related questions are studied for rank jumps and growth of ranks Tate-Shafarevich groups on average in prime power cyclic extensions.
\end{abstract}

\maketitle
\section{Introduction}
\par In \cite{mazurrubin}, B.~Mazur and K.~Rubin introduced the notion of \emph{diophantine stability} for an irreducible variety $V$ defined over a number field $K$. Let $p$ be a prime number. Given a number field extension $L/K$, $V$ is said to be diophantine stable in $L$ if $V(L)=V(K)$. The variety $V$ is said to be diophantine stable at the prime $p$ if there are abundantly many cyclic $p$-extensions of $p$-power order in which $V$ is diophantine stable. More precisely, $V$ is diophantine stable at $p$ if for any choice of $n\in \Z_{\geq 1}$ and finite set of primes $\Sigma$ of $K$, $V$ is diophantine stable in infinitely many $\Z/p^n\Z$-extensions $L/K$ in which the primes of $\Sigma$ split. Mazur and Rubin showed that if $A$ is a simple abelian variety defined over $K$ such that all $\bar{K}$-endomorphisms of $A$ are defined over $K$, then, $A$ is diophantine stable at a postive density set of primes of $K$. A similar result is proved for any curve of genus $\geq 1$ defined over a number field.
\par We study diophantine stability for non-CM elliptic curves $E$ defined over the rationals with good ordinary reduction at $p$. Given a pair $(E,p)$, we ask if $E$ is diophantine stable at $p$. From a statistical point of view, we are interested in the following questions.
\begin{question}
\begin{enumerate}
    \item Given an elliptic curve $E_{/\Q}$, what can be said about proportion of primes $p$ at which $E$ is diophantine stable at $p$?
    \item Given a prime $p$, what can be said about proportion of elliptic curves $E_{/\Q}$ ordered by height that are diophantine stable at $p$?
\end{enumerate}
\end{question}
Inspired by the notion of diophantine stability (for the Mordell-Weil group), we introduce an analogous notion for the $p$-primary part of the Tate-Shafarevich group, see Definition \ref{shastab}. Given an elliptic curve $E_{/\Q}$, we shall assume throughout this paper that $\Sha(E/\Q)$ is finite. Let $p$ be a prime number and $L/\Q$ a cyclic $p$-extension. We say that $E$ is $\Sha$-stable in $L/\Q$ if $\#\Sha(E/L)[p^\infty]=\#\Sha(E/\Q)[p^\infty]$. Further, $E$ is said to be $\Sha$-stable at $p$ if for any $(n, \Sigma)$, there are infinitely many $\Z/p^n\Z$-extensions in which the primes of $\Sigma$ split and in which $E$ is $\Sha$-stable.
\begin{lthm}[Theorem \ref{thm section 4 main}]\label{thmA}
Let $E_{/\Q}$ be an elliptic curve without complex multiplication such that $\op{rank} E(\Q)=0$. Then, $E$ is diophantine stable and $\Sha$-stable at $100\%$ of the primes $p$.
\end{lthm}
Given an elliptic curve $E$ satisfying the above conditions, it is shown in earlier work (see \cite[Theorem A]{BKR}) that for $100\%$ of the primes $p$, $E(L)=E(\Q)$ for infinitely many $\Z/p\Z$-extensions $L/\Q$. The above result is much stronger, in fact for any pair $(n, \Sigma)$ the method gives an explicit recipe to construct extensions for which $E(L)=E(\Q)$, see Remark \ref{remark}. One is able to prove a refinement of Theorem \ref{thmA} when $E(\Q)_{\op{tors}}\neq 0$.
\begin{lthm}[Theorem \ref{4.15}]Let $E_{/\Q}$ be an elliptic curve without complex multiplication such that $E(\Q)_{\op{tors}}\neq 0$ and $\op{rank} E(\Q)=0$. Then $E$ is diophantine stable \emph{and} $\Sha$-stable at all but a finite set of primes $p$ at which it has good ordinary reduction.
\end{lthm}
Furthermore, the finite set of primes outside which the results apply is made explicit. For instance, the elliptic curve $X_0(14)$ is diophantine and $\Sha$-stable at all primes $p\geq 11$ at it has good ordinary reduction.
\par Next, we study the growth of ranks and Tate-Shafarevich groups in cyclic $p$-extensions. Independent results are proved for elliptic curves of rank 0 and positive rank.

\begin{lthm}[Theorem \ref{5.1 main theorem}]
Let $E$ be an elliptic curve without complex multiplication such that $\op{rank} E(\Q)=0$. Then, for $100\%$ of the primes $p$, there are abundantly many $p$-cyclic extensions in which the $p$-primary Selmer group grows. More precisely, for any $n\in \Z_{\geq 1}$ and a finite set of primes $\Sigma$, there are infinitely many $\Z/p^n\Z$-extensions $L/\Q$ such that 
\begin{enumerate}
    \item all primes of $\Sigma$ split in $L$,
    \item $\op{Sel}_{p^\infty}(E/\Q)=0$ and $\op{Sel}_{p^\infty}(E/\Q)\neq 0$.
\end{enumerate}
\end{lthm}
Given an elliptic curve $E_{/\Q}$ without complex multiplication, let $\mathcal{P}_E$ be the set of primes satisfying the conditions of Definition \ref{def 5.1}. These are precisely the set of primes $p$ at which $E$ has good ordinary reduction for which the $\mu$-invariant vanishes and the $\lambda$-invariant is equal to $\op{rank}E(\Q)$. It is expected that $\mathcal{P}_E$ consists of $100\%$ of the primes, according to \cite[Conjecture 1.1]{kunduray1}.
\begin{lthm}[Theorem \ref{5.1 theorem pos rank}]
Let $E$ be an elliptic curve without complex multiplication such that $\op{rank} E(\Q)>0$ and let $p\in \mathcal{P}_E$. Let $L$ be \textit{any} cyclic $p$-extension. Then, at least one of the following assertions hold:
\begin{enumerate}
    \item $\op{rank}E(L)\geq [L:\Q]\op{rank} E(\Q)$,
    \item the $p$-adic regulator over $\Q$ is a unit in $\Z_p$, and over $L$, it is not a unit in $\Z_p$,
    \item
    $\Sha(E/\Q)[p^\infty]=0$ and $\Sha(E/L)[p^\infty]\neq 0$.
\end{enumerate}
\end{lthm}
Next, we prove stability results for a fixed prime $p\geq 11$ and varying elliptic curve $E_{/\Q}$ of Mordell-Weil rank 0. We assume a version of the rank distribution conjecture. Delaunay's heuristics \cite{delaunay} for the Tate-Shafarevich group give a precise estimate for the proportion of elliptic curves $E_{/\Q}$ such that $E$ has rank 0 and $p|\#\Sha(E/\Q)$.
\begin{lthm}[Theorem \ref{last thm}]
Let $p\geq 11$ be a prime and assume that Conjectures \ref{RDC} and \ref{Del} are satisfied. Let $\mathscr{E}_p$ be the set of isomorphism classes of elliptic curves $E_{/\Q}$ satisfying the following conditions:
\begin{enumerate}
    \item $\op{rank} E(\Q)=0$,
    \item $E[p]$ is irreducible as a Galois module,
    \item $E$ has good reduction at $2,3$,
    \item $E$ has good ordinary reduction at $p$,
    \item $E$ is diophantine stable and $\Sha$-stable at $p$.
\end{enumerate}
Then, for the effective constant $c_1>0$ (see Lemma \ref{lemma 6.5}) we have that the lower density of $\mathscr{E}_p$ (in the set of all elliptic curves $E_{/\Q}$) satisfies the lower bound
\[\geq \frac{1}{6}-\frac{1}{p}-\frac{1}{2}\left(1-\prod_i \left(1-\frac{1}{p^{2i-1}}\right)\right)-(\zeta(p)-1)-\zeta(10) c_1\frac{\log p\cdot (\log \log p)^2}{\sqrt{p}}.\]
\end{lthm}
\par In particular, the lower density approaches $\frac{1}{6}$ as $p\rightarrow \infty$.
\par We apply results in Iwasawa theory to prove results on the rank jumps and growth of Tate-Shafarevich groups in cyclic $p$-extensions. The initial conception of this idea can be traced back to the work of T.~Dokchitser, see \cite{Dok}, in which the rank jump of elliptic curves in cubic extensions is studied. This is done by applying a formula of Y.~Hachimori and K.~Matsuno (see \cite{hachimorimatsuno}) for the growth of the Iwasawa $\lambda$-invariant in cyclic $p$-extensions.
\par We extend this approach to show that there is a close relationship between behaviour of Iwasawa invariants and diophantine stability. In previous work \cite{kunduray1}, results are proved for Iwasawa invariants on average. These methods are significantly extended to prove the results in this paper.
\subsection*{Acknowledgements} I would like to thank Debanjana Kundu, Ravi Ramakrishna and Tom Weston for helpful discussions.
\section{Preliminaries}
\label{section: preliminaries}
\subsection{}
\par Let $E_{/\Q}$ be an elliptic curve of conductor $N$ and $p$ an odd prime. We introduce main object of study in Iwasawa study, namely the $p$-primary Selmer group over the cyclotomic $\Z_p$-extension. For each integer, $n\in \Z_{\geq 1}$, denote by $E[n]$ the subgroup of $n$-torsion points of $E(\bar{\Q})$. Set $E[p^\infty]$ to denote the $p$-primary torsion subgroup of $E(\bar{\Q})$ given by the union
\[E[p^\infty]:=\bigcup_{k\geq 1} E[p^k].\]
Let $S$ be a set of primes containing those dividing $Np$, and $\Q_S\subset \bar{\Q}$ the maximal extension of $\Q$ in which all primes $\ell\notin S$ are unramified. Note that $E[p^\infty]$ admits an action of $\op{Gal}(\Q_S/\Q)$. Given a number field extension $F$ of $\Q$ contained in $\Q_S$, we set \[H^1(\Q_S/F, E[p^\infty]):=H^1(\op{Gal}(\Q_S/F), E[p^\infty]).\]
The $p$-primary Selmer group for $E$ over $F$ consists of global cohomology classes subject to local conditions, defined as follows
\[\op{Sel}_{p^\infty}(E/F):=\op{ker}\left(H^1(\Q_S/F, E[p^\infty])\longrightarrow \bigoplus_{v} H^1(F_v, E)[p^\infty]\right),\]where in the above sum, $v$ ranges over all finite primes of $F$. 
\par For $n\geq 0$, let $\Q_n$ be the subfield of $\Q(\mu_{p^{n+1}})$ degree $p^n$ over $\Q$. Let $F_n$ denote the composite $F_n:=F\cdot \Q_n$, and note that $F_n$ is contained in $F_{n+1}$. Let $F_{\infty}$ be the union \[F_{\infty}:=\bigcup_{n\geq 0} F_n\]and set $\Gamma_F:=\op{Gal}(F_{\infty}/F)$. Note that there are isomorphisms of topological groups \[\op{Gal}(F_{\infty}/F)\xrightarrow{\sim} \varprojlim_n\op{Gal}(F_{n}/F)\xrightarrow{\sim} \Z_p.\] The extension $F_{\infty}$ is the cyclotomic $\Z_p$-extension of $F$ and $F_n$ is its \textit{$n$-th layer}. Choose a topological generator $\gamma_F\in \Gamma_F$ and fix an isomorphism $\Z_p\xrightarrow{\sim} \Gamma_F$ sending $a$ to $\gamma_F^a$. The Iwasawa algebra $\Lambda(\Gamma_F)$ is defined as the following inverse limit
\[\Lambda(\Gamma_F):=\varprojlim_n \Z_p[\op{Gal}(F_n/F)].\] Fix an isomorphism of $\Lambda(\Gamma_F)$ with the ring of formal power series $\Z_p\llbracket x\rrbracket$, by identifying $\gamma_F-1$ with $x$.
\par Assume that $F$ is a number field extension of $\Q$ such that $F\cap \Q_\infty=\Q$. Then, there is a natural isomorphism $\Gamma_\Q\xrightarrow{\sim} \Gamma_F$ and thus we may identify $\Lambda(\Gamma_F)$ with $\Lambda:=\Lambda(\Gamma_\Q)$. Consider the direct limit taken with respect to restriction maps
\[\op{Sel}_{p^\infty} (E/F_\infty):=\varinjlim_n \op{Sel}_{p^\infty}(E/F_n).\]
The Pontryagin dual $\op{Sel}_{p^\infty} (E/F_\infty)^{\vee}$ is a finitely generated $\Lambda$-module and is the main object of interest in the Iwasawa theory of elliptic curves. When $E$ has good ordinary reduction at $p$, it is a deep result of Kato \cite{katozeta} that the Pontryagin dual \[\op{Sel}_{p^\infty} (E/F_\infty)^{\vee}:=\op{Hom}_{\op{cnts}}\left(\op{Sel}_{p^\infty} (E/F_\infty), \Q_p/\Z_p\right)\] is a torsion $\Lambda$-module.
\subsection{}
\par Let $M$ be a cofinitely generated cotorsion $\Z_p\llbracket x\rrbracket$-module, i.e., the Pontryagin-dual $M^{\vee}:=\op{Hom}(M, \Q_p/\Z_p)$ is a finitely generated and torsion $\Z_p\llbracket x\rrbracket$-module. Recall that a polynomial $f(x)\in \Zx$ is said to be \textit{distinguished} if it is a monic polynomial whose non-leading coefficients are all divisible by $p$. Note that all height $1$ prime ideals of $\Z_p\llbracket x\rrbracket$ are principal ideals $(a)$, where $a=p$ or $a=f(x)$, where $f(x)$ is an irreducible distinguished polynomial. 
According to the structure theorem for $\Zx$-modules (see \cite[Theorem 13.12]{washington1997}), $M^{\vee}$ is pseudo-isomorphic to a finite direct sum of cyclic $\Zx$-modules, i.e., there is a map
\[
M^{\vee}\longrightarrow \left(\bigoplus_{i=1}^s \Zx/(p^{\mu_i})\right)\oplus \left(\bigoplus_{j=1}^t \Zx/(f_j(T)) \right)
\]
with finite kernel and cokernel.
Here, $\mu_i>0$ and $f_j(T)$ is a distinguished polynomial.
The characteristic ideal of $M^\vee$ is (up to a unit) generated by
\[
f_{M}^{(p)}(T) = f_{M}(T) := p^{\sum_{i} \mu_i} \prod_j f_j(T).
\]
The $\mu$-invariant of $M$ is defined as the power of $p$ in $f_{M}(T)$.
More precisely,
\[
\mu_p(M):=\begin{cases}
\sum_{i=1}^s \mu_i & \textrm{ if } s>0\\
0 & \textrm{ if } s=0.
\end{cases}
\]
The $\lambda$-invariant of $M$ is the degree of the characteristic element, i.e.,
\[
\lambda_p(M) :=\begin{cases}
\sum_{i=1}^s \deg f_i & \textrm{ if } s>0\\
0 & \textrm{ if } s=0.
\end{cases}
\]
%It is easy to see that $\lambda_p(M)=\op{rank}_{\Z_p} M^\vee$.
\subsection{} In the rest of this section, assume that $E$ is an elliptic curve defined over $\Q$ with good ordinary reduction at an odd prime $p$. Let $L/\Q$ be a cyclic extension with $\op{Gal}(L/\Q)\simeq \Z/p^n\Z$ for some integer $n\in \Z_{\geq 1}$. We shall refer to such an extension of $\Q$ as a $p$-cyclic extension of $\Q$. Kato showed that the Selmer group $\op{Sel}_{p^\infty}(E/\Q_\infty)$ is a cofinitely generated and cotorsion $\Lambda$-module. Let $\mu_p(E/\Q)$ and $\lambda_p(E/\Q)$ denote the Iwasawa $\mu$ and $\lambda$-invariants of $\op{Sel}_{p^\infty}(E/\Q_\infty)$ respectively. Likewise, let $\mu_p(E/L)$ and $\lambda_p(E/L)$ be the $\mu$ and $\lambda$-invariants of $\op{Sel}_{p^\infty}(E/L_\infty)$ respectively. The following conjecture of R.Greenberg is wide open.

\begin{Conjecture}\label{mu=0 conjecture} \cite[Conjecture 1.11]{greenbergITEC}
Let $E_{/\Q}$ be an elliptic curve and $p$ a prime at which $E$ has good ordinary reduction. Assume that $E[p]$ is irreducible as a module over $\op{Gal}(\bar{\Q}/\Q)$, then, $\mu_p(E/\Q)=0$.
\end{Conjecture}The hypothesis that $E[p]$ is irreducible is indeed necessary, since examples for which $\mu_p(E/\Q)>0$ were constructed by Mazur in \cite{mazurrationalpoints}, when $E[p]$ is reducible as a Galois module. In \cite{raymuinvariant}, it is shown that the Galois modules for which $E[p]$ is reducible may be classified into two types, namely \textit{aligned} and \textit{skew}. In the case when $E[p]$ is aligned, it is shown that $\mu_p(E/\Q)>0$ and in the case when $E[p]$ is skew, examples indicate that $\mu_p(E/\Q)=0$.
\par From a statistical point of view, one is interested in how often the Iwasawa $\mu$-invariant vanishes. The following implication of Greenberg's conjecture follows from well known results.
\begin{Proposition}
Assume Greenberg's conjecture for all pairs $(E,p)$, where $E$ is an elliptic curve defined over $\Q$ and $p$ is a prime at which $E$ has good ordinary reduction. Then, the following assertions hold.
\begin{enumerate}
    \item If $E_{/\Q}$ is a semistable elliptic curve without complex multiplication, then, $\mu_p(E/\Q)=0$ for all primes $p\geq 13$ at which $E$ has good ordinary reduction.
    \item For a fixed elliptic curve $E_{/\Q}$ without complex-multiplication, for all but finitely many primes $p$ at which $E$ has good ordinary reduction, $\mu_p(E/\Q)=0$.
    \item For $100\%$ of elliptic curves $E_{/\Q}$, the $\mu$-invariant $\mu_p(E/\Q)=0$ for all primes $p$ at which $E$ has good ordinary reduction.
\end{enumerate}
\end{Proposition}
\begin{proof}
\par Let $E_{/\Q}$ be a semistable elliptic curves without complex mutliplication. Mazur showed that $E[p]$ is irreducible as a Galois module for all primes $p\geq 13$.
\par Let $E$ be a fixed elliptic curve without complex-multiplication. It follows from Serre's renowned open image theorem \cite{serreopenimage} that the residual representation \[\bar{\rho}_{E,p}:\op{Gal}(\bar{\Q}/\Q)\rightarrow \op{GL}_2(\F_p)\] for the action of the absolute Galois group on $E[p]$ is surjective for all but finitely many primes. As a consequence, the Galois representation on $E[p]$ is irreducible for all but finitely many primes, and hence, Greenberg's conjecture imples that $\mu_p(E/\Q)=0$ for all but finitely many primes $p$.
\par It is shown by W.Duke \cite{duke} that for $100\%$ of elliptic curves defined over $\Q$, the residual representation on $E[p]$ is irreducible for all but finitely many primes. Hence, Greenberg's conjecture implies that $\mu_p(E/\Q)=0$ for all but finitely many primes $p$ at which $E$ has good ordinary reduction.
\end{proof}
According to the next result of Hachimori-Matsuno \cite{hachimorimatsuno}, the $\mu=0$ property is stable in $p$-cyclic extensions. Furthermore, the $\lambda$-invariants $\lambda_p(E/\Q)$ and $\lambda_p(E/L)$ are related by a formula which generalizes the formula of Kida in \cite{kida}. Let $E_{/\Q}$ be an elliptic curve with good ordinary reduction at $p$ and $L/\Q$ a $p$-cyclic extension.  In what follows $\tilde{w}$ will be a prime of $L_\infty$  and $\tilde{\ell}$ the prime below it in $\Q_\infty$. Set $e(\tilde{w}/\tilde{\ell})$ to denote the ramification index of $\tilde{w}$ over $\tilde{\ell}$. Let $P_1$ (resp. $P_2$) be the set of primes $\tilde{w}$ at which $E$ has split multiplicative reduction (resp. $E(L_{\infty, \tilde{w}})$ has a point of order $p$).
\begin{Th}[Y.Hachimori-K.Matsuno]\label{kida thm}
Let $E_{/\Q}$ be an elliptic curve and $p$ an odd prime at which $E$ has good ordinary reduction. Let $L/\Q$ be a $p$-cyclic extension and that $L\cap \Q_\infty= \Q$. Assume $\mu_p(E/\Q)=0$, then the Selmer group $\op{Sel}_{p^\infty}(E/L_\infty)$ is cotorsion as a $\Lambda$-module and $\mu_p(E/L)=0$. Furthermore, the $\lambda$-invariants are related according to the following formula
\begin{equation}\label{kidaformula}\lambda_p(E/L)=[L:\Q]\lambda_p(E/\Q)+\sum_{\tilde{w}\in P_1} \left(e(\tilde{w}/\tilde{\ell})-1\right)+2\sum_{\tilde{w}\in P_2} \left(e(\tilde{w}/\tilde{\ell})-1\right).\end{equation}
\end{Th}
\begin{proof}
The above result is \cite[Theorem 3.1]{hachimorimatsuno}.
\end{proof}
\section{The truncated Euler characteristic}
\par In this section, we recall the notion of the truncated Euler characteristic and its relationship with Iwasawa invariants. The reader may consult \cite{csslinks, zerbes, ray1, raymulti} for a more detailed introduction to the topic. Let $E_{/\Q}$ be an elliptic curve and $p$ an odd prime at which $E$ has good ordinary reduction. Let $L/\Q$ be a $p$-cyclic extension and assume that $L\cap \Q_\infty=\Q$. Note that there is a canonical isomorphism $\Gamma_\Q\simeq \Gamma_L$, with respect to which $\Lambda=\Lambda(\Gamma_\Q)$ is identified with $\Lambda(\Gamma_F)$. According to a result of K.Kato and K.Rubin, the Selmer group $\op{Sel}_{p^\infty}(E/L_\infty)$ is cotorsion as a $\Lambda$-module. In this section, let $\Gamma:=\Gamma_\Q$ and identify $\Gamma$ with $\Gamma_L$. Let $F$ be either $\Q$ or $L$, and consider the natural map
\[\Phi: \op{Sel}_{p^\infty} (E/F_\infty)^\Gamma\rightarrow \op{Sel}_{p^\infty} (E/F_\infty)_{\Gamma}\]
sending an element $x\in \op{Sel}_{p^\infty} (E/F_\infty)^\Gamma$ to its residue class in $\Phi(x)=\bar{x}\in \op{Sel}_{p^\infty} (E/F_\infty)_{\Gamma}$. If $\op{rank} E(F)\leq 1$, the kernel and cokernel of $\Phi$ are both finite. Moreover, when $\op{rank} E(\Q)>1$, there is an explicit criterion for the kernel and cokernel of $\Phi$ to be finite, see \cite[Lemma 3.3]{raymulti}.
\begin{Definition}
The truncated Euler characteristic $\chi_t(\Gamma_F, E[p^\infty])$ is defined as the following quotient
\[\chi_t(\Gamma_F, E[p^\infty]):=\frac{\# \op{ker}\Phi}{\#\op{cok} \Phi},\] provided the kernel and cokernel of $\Phi$ are both finite.
\end{Definition}
Let $f(T)$ denote the characteristic element of $\op{Sel}_{p^\infty}(E/F_\infty)$. Let $r\in \Z_{\geq 0}$ denote the order of vanishing of $f(T)$ at $T=0$, and write $f(T)=T^r\cdot g(T)$. Note that $g(0)\neq 0$. Let $\op{Reg}_p(E/F)$ denote the $p$-adic regulator of $E$ over $F$, defined as the determinant of the $p$-adic height pairing, see \cite{schneiderhp1}. Note that the $p$-adic regulator is conjectured to be non-zero, see \cite{Schneider85}. Let $\mathcal{R}_p(E/F)$ denote the normalized regulator, defined as $\mathcal{R}_p(E/F):=p^{-\op{rank}E(F)}\op{Reg}_p(E/F)$. Given $a, b\in \Q_p$, we write $a\sim b$ if $a=ub$ for a $p$-adic unit $u\in \Z_p^{\times}$.
\par The following result gives the formula for the truncated Euler characteristic of the $p$-primary Selmer group when it is defined.
 In the CM case, this was proven by B. Perrin-Riou (see \cite{PR82}) and in the general case by P. Schneider (see \cite{Schneider85}).
\begin{Th}\label{th ECF}
Let $E$ be an elliptic curve over a number field $F$ and $p$ an odd prime. Assume that the following conditions are satisfied:
\begin{enumerate}
    \item\label{th ECF c1} $E$ has good ordinary reduction at all primes $v|p$,
    \item\label{th ECF c2} the $p$-primary part of the Tate-Shafarevich group $\Sha(E/F)[p^\infty]$ is finite,
    \item\label{th ECF c3} the $p$-adic regulator $\op{Reg}_p(E/F)$ is non-zero,
    \item\label{th ECF c4} $E(F)[p]=0$.
\end{enumerate}
Then, the order of vanishing of $f(T)$ at $T=0$ is equal to $\op{rank} E(F)$ and 
\[g(0)\sim \mathcal{R}_p(E/F) \times \# \Sha(E/F)[p^\infty] \times \prod_{v\nmid p} c_{v}^{(p)}(E/F) \times \left(\# \widetilde{E}(\kappa_v)[p^\infty]\right)^2,\]where we recall that $\mathcal{R}_p(E/F):=\frac{\op{Reg}_p(E/F)}{p^{\op{rank}E(F)}}$.
\end{Th}
\begin{Remark}
When the rank of $E(F)$ is 0, the term $\mathcal{R}_p(E/F)=1$ by definition, hence there are only three contributing terms to the above formula in this case. Note that if $E$ is an elliptic curve over $\Q$ such that $E(\Q)[p]=0$, then $E(L)[p]=0$ for any $p$-cyclic extension $L/\Q$. Thus the condition \eqref{th ECF c4} is satisfied for $E_{/L}$.
\end{Remark}
\begin{Proposition}
Let $E_{/F}$ be such that the conditions of Theorem \ref{th ECF} are satisfied and furthermore, the kernel and cokernel of the map $\Phi$ are finite. Then, we have that
\[\chi_t(\Gamma_F, E[p^\infty])\sim g(0).\]
\end{Proposition}

The above result motivates the following definition.
\begin{Definition}
Let $E_{/F}$ be an elliptic curve satisfying the conditions of Theorem \ref{th ECF}. Suppose that $\op{rank} E(F)>1$ and that the truncated Euler characteristic is not defined. Then, we simply set $\chi_t(\Gamma_F, E[p^\infty])$ to denote $|g(0)|_p^{-1}$.
\end{Definition}
Thus, from here on in, the truncated Euler characteristic is always defined and equals $|g(0)|_p^{-1}$. Putting it all together, we have the following result.
\begin{Corollary}\label{cor ECF}
Let $E$ be an elliptic curve over a number field $F$ and $p$ an odd prime. Assume that the conditions of Theorem \ref{th ECF} and satisfied. Then, the truncated Euler characteristic is given (up to a $p$-adic unit) by
\[\chi_t(\Gamma_F, E[p^\infty])\sim  \mathcal{R}_p(E/F) \times \# \Sha(E/F)[p^\infty] \times \prod_{v\nmid p} c_{v}^{(p)}(E/F) \times \left(\# \widetilde{E}(\kappa_v)[p^\infty]\right)^2.\]
\end{Corollary}

The truncated Euler characteristic $\chi_t(\Gamma_F, E[p^\infty])$ is of the form $p^n$, where $n\in \Z_{\geq 0}$. Thus, either $\chi_t(\Gamma_F, E[p^\infty])=1$ or divisible by $p$. When $\chi_t(\Gamma_F, E[p^\infty])=1$, the Iwasawa invariants $\mu_p(E/F)$ and $\lambda_p(E/F)$ are fully understood. The following result gives a relationship between the Iwasawa invariants and the (truncated) Euler characteristic.
\begin{Lemma}\label{lemma ECF mulambda}
Let $E_{/F}$ be an elliptic curve satisfying the conditions of Theorem \ref{th ECF}. Then, the following are equivalent
\begin{enumerate}
    \item $\mu_p(E/F)=0$ and $\lambda_p(E/F)=\op{rank} E(F)$,
    \item $\chi_t(\Gamma_F, E[p^\infty])=1$.
\end{enumerate}
\end{Lemma}
\begin{proof}
The result follows from \cite[Lemma 3.4]{kunduray1}.
\end{proof}
\section{Stability for a fixed elliptic curve of rank 0 and varying prime $p$}\label{s4}
\par In this section, we fix an odd prime number $p$ and an elliptic curve $E_{/\Q}$ with good ordinary reduction at $p$. Recall that a cyclic $p$-extension $L/\Q$ is a Galois extension for which $\op{Gal}(L/\Q)\simeq \Z/p^n\Z$ for some positive integer $n$. We shall also refer to such an extension as a $\Z/p^n\Z$-extension. Let $\op{Cyc}_p$ denote the set of all cyclic $p$-extensions $L/\Q$ and for $n\in \Z_{\geq 1}$ denote by $\op{Cyc}_p^n\subset \op{Cyc}_p$ the subset of cyclic $p$-extensions $L/\Q$ such that $\op{Gal}(L/\Q)\simeq \Z/p^n\Z$. Given a finite set of prime numbers $\Sigma$, let $\op{Cyc}_p^{\Sigma, n}$ consist of $L\in \op{Cyc}_p^{n}$ such that all primes $\ell \in \Sigma$ split in $L$.
\begin{Definition}\label{def diophantine stab}
Let $V$ be a variety defined over $\Q$. For a number field extension $L/\Q$, $V$ is said to be \textit{diophantine stable} w.r.t. $L$ if $V(L)=V(\Q)$. Let $p$ be a prime number, $V$ is said to be diophantine stable at $p$ if for every positive integer $n$ and every finite set of primes $\Sigma$, there are infinitely many cyclic $p$-extensions $L\in \op{Cyc}_p^{\Sigma, n}$ such that $V$ is diophantine stable w.r.t. $L$.

\end{Definition}

We note that the additional requirement that any finite set of primes $\Sigma$ should be split is quite a strong one. One is naturally interested in the following question in arithmetic statistics. 
\begin{question}
Given a variety $V_{/\Q}$, what is the proportion of primes $p$ such that it can be shown that $V$ is diophantine stable at $p$?
\end{question}

For a simple abelian variety $A_{/\Q}$ whose $\bar{\Q}$-endomorphisms are all defined over $\Q$, then Mazur and Rubin show that there is a set of primes $p$ of positive density at which $A$ is diophantine stable at $p$. It is shown in Theorem \ref{thm section 4 main} that if $E_{/\Q}$ is a non-CM elliptic curve with Mordell-Weil rank 0, then $E$ is diophantine stable at $p$ for consists of $100\%$ of the primes $p$. This particular result should be contrasted to an earlier result \cite[Theorem A]{BKR}, in which it is shown that for $100\%$ of the primes $p$, there are infinitely many $\Z/p\Z$-extensions $L/\Q$ in for which $E(\Q)=E(L)$. However, the result proved in this section is much stronger since it gives more control on the splitting behaviour of primes in $p$-cyclic $\Z/p^n\Z$-extensions. According to Remark \ref{remark after thm}, the set of all extensions $L\in \op{Cyc}_p^{\Sigma, n}$ w.r.t. which $E$ is diophantine stable has a an explicit description, which makes them very abundant. Furthermore, Theorem \ref{thm section 4 main} asserts that the stability of the $p$-primary part of the Tate-Shafarevich can also be controlled for $100\%$ of the primes $p$. Inspired by the definition of Mazur and Rubin for diophantine stability of the Mordell Weil group, we introduce the notion of $\Sha$-stability, for the $p$-primary part of the Tate-Shafarevich group.
\begin{Definition}\label{shastab}
 Let $E_{/\Q}$ be an elliptic curve for which $\Sha(E/\Q)$ is finite. Let $p$ be a prime number. We say that $E$ is $\Sha$-stable at $p$ if for every $n\in \Z_{\geq 1}$ and finite set of primes $\Sigma$, there are infinitely many $L\in \op{Cyc}_p^{n,\Sigma}$ such that \begin{equation}\label{sha stable}\#\Sha(E/L)[p^\infty]=\#\Sha(E/\Q)[p^\infty].\end{equation}
\end{Definition}
\par Note that since $\Sha(E/\Q)$ is assumed to be finite, it follows that $\Sha(E/\Q)[p^\infty]=0$ for all but finitely many primes $p$. Thus for all but finitely many primes $p$, \eqref{sha stable} translates to the condition $\Sha(E/L)[p^\infty]=0$.

\begin{Definition}\label{def PE}For an elliptic curve $E_{/\Q}$ for which $\op{rank} E(\Q)=0$, let $\mathcal{P}_E$ consist of the primes $p$ such that
\begin{enumerate}
    \item $p$ is odd,
    \item $E[p]$ is irreducible as a $\op{Gal}(\bar{\Q}/\Q)$-module,
    \item $E$ has good ordinary reduction at $p$,
    \item $\op{Sel}_{p^\infty}(E/\Q_\infty)=0$.
\end{enumerate}
\end{Definition}
The following result will be used at various points in in this paper.
\begin{Lemma}\label{NSW lemma}
Let $G$ and $M$ be finite abelian groups of $p$-power order such that $G$ acts on $M$. Suppose that $M^G=0$, then $M=0$.
\end{Lemma}
\begin{proof}
The result follows from \cite[Proposition 1.6.12]{NSW}.
\end{proof}
\begin{Remark}\label{remark}
For $p\in \mathcal{P}_E$, we have that $E(\Q)[p]=0$ since $E[p]$ is irreducible as a Galois module. Since $\Q_\infty$ is a pro-$p$ extension of $\Q$, it follows from Lemma \ref{NSW lemma} that $E(\Q_\infty)[p^\infty]=0$. Therefore, $\op{Sel}_{p^\infty}(E/\Q)$ injects into $\op{Sel}_{p^\infty}(E/\Q_\infty)$. On the other hand, $\op{Sel}_{p^\infty}(E/\Q)$ fits into a short exact sequence
\[0\rightarrow E(\Q)\otimes \Q_p/\Z_p\rightarrow \op{Sel}_{p^\infty}(E/\Q)\rightarrow \Sha(E/\Q)[p^\infty]\rightarrow 0.\] Hence, for $p\in \mathcal{P}_E$, we have that $\Sha(E/\Q)[p^\infty]=0$.
\end{Remark}
For $x>0$, let $\mathcal{P}_E^c(x)$ consist of all primes $p\notin \mathcal{P}_E$ such that $E$ has good ordinary reduction at $p$ and $p\leq x$.
\begin{Th}\label{100 percent}
The set $\mathcal{P}_E$ consists of $100\%$ of the primes $p$. Furthermore, we have the following growth estimate for the complement of $\mathcal{P}_E$
\[\# \mathcal{P}_E^c(x)\ll \frac{x (\log \log x)^2}{(\log x)^2}.\]
\end{Th}
\begin{proof}
The result essentially follows is due to R.Greenberg \cite[Theorem 5.1]{greenbergITEC} and the asymptotic estimate follows from the main result of V.K.Murty \cite{murty}. We provide details here for completeness. It follows from Serre's open image theorem that there $E[p]$ is irreducible as a Galois module for all but finitely many primes. Serre showed that $E$ has good ordinary reduction for $100\%$ of primes $p$, see \cite{serrechebotarev}. It follows from \cite[Corollary 3.6]{kunduray1} and the Euler characteristic formula that the following conditions are equivalent for a prime $p$ at which $E$ has good ordinary reduction and $E[p]$ is irreducible as a Galois module
\begin{enumerate}
    \item $\op{Sel}_{p^\infty}(E/\Q_\infty)=0$, 
    \item $\mu_p(E/\Q)=0$ and $\lambda_p(E/\Q)=0$, 
    \item $\chi_t(\Gamma_\Q, E[p^\infty])=1$,
    \item $p\nmid \#\Sha(E/\Q)$, $p\nmid c_\ell(E)$ for all primes $\ell\neq p$ and $p\nmid \#\widetilde{E}(\F_p)$.
\end{enumerate}
Clearly, $p\nmid \#\Sha(E/\Q)$ and $p\nmid c_\ell(E)$ for all but finitely many primes $p$. It follows from \cite{murty} that the proportion of primes $p$ such that $p\mid \#\widetilde{E}(\F_p)$ is $\ll \frac{x(\log \log x)^2}{(\log x)^2}$. The result follows from this.
\end{proof}

The main result of this section is the following.
\begin{Th}\label{thm section 4 main}
Let $E_{/\Q}$ be an elliptic curve without complex multiplication such that $\op{rank} E(\Q)=0$. Then $E$ is diophantine stable \emph{and} $\Sha$-stable at every prime $p\geq 11$ such that $p\in \mathcal{P}_E$. As a result, $E$ is diophantine stable and $\Sha$-stable at $100\%$ of the primes $p$.\end{Th}
\begin{Remark}\label{remark after thm}
 In fact the method produces many extensions $L/\Q$ with respect to which $E$ is both diophantine and $\Sha$-stable. Let $p\geq 11$ be a prime in $\mathcal{P}_E$ and $n\in \Z_{\geq 1}$. Then, there is a positive density set of primes $\mathfrak{S}_n$ such that given a finite set of primes $\Sigma$, and any set of primes $\ell_1, \dots, \ell_{\#\Sigma+1}\in \mathfrak{S}_n$, there is a $\Z/p^n\Z$-extension $L/\Q$ such that all of the following conditions hold:
\begin{enumerate}
    \item the primes that ramify in $L$ are a subset of $\{\ell_1, \dots, \ell_{\#\Sigma+1}\}$,
    \item all primes of $\Sigma$ split in $L$, 
    \item $E(L)=E(\Q)$, 
    \item $\Sha(E/L)[p^\infty]=0$.
\end{enumerate}
Furthermore, according to Lemma \ref{s4 positive density}, the density of $\mathfrak{S}_n$ is equal to $\frac{p^2-p-1}{p^{n-1} (p+1)(p-1)^2}$.
\end{Remark}
We now make preparations for the proof of Theorem \ref{thm section 4 main}. Given an elliptic curve with good reduction at a prime $\ell$, let $\widetilde{E}$ denote the reduced curve at $\ell$, and $\widetilde{E}(\F_\ell)$ the group of $\F_\ell$-points of $E$.
\begin{Definition}
 Let $E_{/\Q}$ be an elliptic curve and $p$ an odd prime. Let $Q_1=Q_1(E,p)$ be the set of primes $\ell\neq p$ at which $E$ has bad reduction and $Q_2=Q_2(E,p)$ the set of primes $\ell\neq p$ at which $E$ has good reduction and $p$ divides $\#\widetilde{E}(\F_\ell)$.
\end{Definition}
 Note that the set of primes $Q_1$ is finite, since the set of primes at which $E$ has bad reduction is finite. Given a cyclic $p$-extension $L/\Q$, let $\Sigma_L$ be the set of primes $\ell\neq p$ that ramify in $L$. 
 \begin{Lemma}\label{stability lemma}
 Let $E_{/\Q}$ be an elliptic curve with $\op{rank} E(\Q)=0$ and $p$ an odd prime such that
 \begin{enumerate}
     \item $E$ has good ordinary reduction at $p$, 
     \item $E[p]$ is irreducible as a Galois module,
     \item $\op{Sel}_{p^\infty}(E/\Q_\infty)=0$.
 \end{enumerate}
 Let $L/\Q$ be a cyclic $p$-extension such that
 \begin{enumerate}
    \item $L\cap \Q_\infty=\Q$,
     \item $\Sigma_L$ and $Q_1\cup Q_2$ are disjoint.
 \end{enumerate} Then, the following assertions hold
 \begin{enumerate}
     \item\label{stability: 1} $\op{rank} E(L)=0$,
     \item\label{stability: 2} $\Sha(E/L)[p^\infty]=\Sha(E/\Q)[p^\infty]=0$.
     \item\label{stability: 3} If $p\geq 11$, then $E$ is diophantine-stable w.r.t. $L$, i.e., $E(L)=E(\Q)$.
 \end{enumerate}
 \end{Lemma}
 \begin{proof}
 First, we prove \eqref{stability: 1}. Let $\ell\in \Sigma_L$, and $\tilde{w}$ a prime of $L_\infty$ above $\ell$. Since $\ell\notin Q_1$, the elliptic curve $E$ has good reduction at $\ell$. Therefore, $\tilde{w}$ is not contained in $P_1$. Since $\ell\notin Q_2$, it follows that $\widetilde{E}(\F_\ell)[p]=0$. Let $k_{\tilde{w}}$ denote the residue field at $\tilde{w}$, since $k_w/\F_\ell$ is pro-$p$ it follows that $\widetilde{E}(k_{\tilde{w}})[p]=0$. Note that the reduction map
 \[E(L_{\infty,\tilde{w}})\rightarrow \widetilde{E}(k_{\tilde{w}})\] is surjective (since $E$ has good reduction at $\ell$) and its kernel is pro-$\ell$. It follows that the map on $p$-torsion
 \[E(L_{\infty,\tilde{w}})[p]\rightarrow \widetilde{E}(k_{\tilde{w}})[p]\]is an isomorphism. Therefore, $E(L_{\infty,\tilde{w}})$ has no $p$-torsion point and we have shown that $\tilde{w}\notin P_1\cup P_2$. Since the Selmer group $\op{Sel}_{p^\infty}(E/\Q_\infty)=0$, we have that
 \[\mu_p(E/\Q)=0\text{ and }\lambda_p(E/\Q)=0.\] Therefore, it follows from \eqref{kidaformula} that $\lambda_p(E/L)=[L:\Q]\lambda_p(E/\Q)=0$. On the other hand, according to \cite[Theorem 1.9]{greenbergITEC} 
 \[\op{rank} E(L)\leq \lambda_p(E/L),\] and we deduce that $\op{rank} E(L)=0$.
 \par Next, we prove \eqref{stability: 2}. Since the Selmer group is $0$, we have that $\mu_p(E/\Q)=0$. It follows from Theorem \ref{kida thm} that $\mu_p(E/L)=0$ as well. On the other hand, it has been shown in the previous paragraph that $\lambda_p(E/L)=0$. Note that since $\op{rank} E(L)=0$, we have that $\mathcal{R}_p(E/L)=1$, and thus the hypotheses of Lemma \ref{lemma ECF mulambda} are satisfied. Therefore, it follows from Lemma \ref{lemma ECF mulambda} that $\chi_t(\Gamma_L, E[p^\infty])=1$. Since $E(\Q)[p]=0$ and $L/\Q$ is a $p$-cyclic extension, it follows from Lemma \ref{NSW lemma} that $E(F)[p]=0$. Therefore, according to the Euler characteristic formula
 \[\chi_t(\Gamma_L, E[p^\infty])\sim  \# \Sha(E/L)[p^\infty] \times \prod_{v\nmid p} c_{v}^{(p)}(E/L) \times \prod_{v|p}\left(\# \widetilde{E}(\kappa_v)[p^\infty]\right)^2.\]
 Since the Euler characteristic is $1$, it follows that each of the terms in the above formula are equal to $1$. In particular, we have shown that $\Sha(E/L)[p^\infty]=0$.
 \par It has been shown that $\op{rank} E(L)=0$. Therefore, in order to prove \eqref{stability: 3}, it suffices to observe that $E(L)_{\op{tors}}=E(\Q)_{\op{tors}}$. This assertion follows from \cite[Theorem 7.2]{GJN}.
 \end{proof}
 \begin{Lemma}\label{s4 splitting}
 Let $\Sigma$ be a finite set of rational primes and $n\in \Z_{\geq 1}$. Set $t:=\# \Sigma$. Let $\ell_1, \dots, \ell_{t+1}$ be primes such that $\ell_i\notin \Sigma$ and $\ell_i\equiv 1\mod{p^n}$ for all $i$ and $m:=\prod_{i=1}^{t+1} \ell_i$. There is a $\Z/p^n\Z$-extension $L/\Q$ contained in $\Q(\mu_m)$ in which all primes $q\in \Sigma$ split.
 \end{Lemma}
 \begin{proof}
 Note that there is a Galois extension $E/\Q$ contained in $\Q(\mu_m)$ such that 
 \[\op{Gal}(E/\Q)\simeq \left(\Z/p^n \Z \right)^{t+1}.\] Consider the subgroup $H$ of $\op{Gal}(E/\Q)$ generated by the Frobenius elements at the $t$ primes in $\Sigma$. We show that there is a quotient of $\op{Gal}(E/\Q)/H$ isomorphic to $\Z/p^n\Z$. From this, it shall follow that there is a $\Z/p^n\Z$-extension $L/\Q$ contained in $E$ in which all primes $q\in \Sigma$ split.
 \par Set $G:=\op{Gal}(E/\Q)\simeq \left(\Z/p^n \Z \right)^{t+1}$ and for $k>0$, let $G_k$ be the subgroup generated by $p^k$-th multiples of the generators of $G$. In other words, $G_k$ corresponds to the subgroup $\left(p^k\Z/p^n \Z \right)^{t+1}$ of $\left(\Z/p^n \Z \right)^{t+1}$. Set $H_k:=G_k\cap H$, and note that $H_{n-1}$ is a $\Z/p\Z$-vector space of dimension $\leq t$. On the other hand, $G_{n-1}$ is a $\Z/p\Z$-vector space of dimension $t+1$. Let $N$ be a subspace of $G_{n-1}$ of dimension $t$ containing $H_{n-1}$, and let $\tilde{N}$ be the subgroup of $G$ consisting of all elements $g\in G$ such that $p^{n-1} g$ is contained in $N$. It is easy to see that $\tilde{N}$ contains $H$ and the quotient $G/\tilde{N}$ is isomorphic to $\Z/p^n\Z$.
 \end{proof}
 % Pick generators of N and extend by another to G_{n-1}. Then take p^{n-1}-roots of these generators. They will be linearly independent over z/p^n. Hence, the last generator will give a z/p^n quotient.
\begin{Definition}\label{definition sn} Let $E_{/\Q}$ be an elliptic curve of conductor $N$ and $p$ a prime. For $n\in \Z_{\geq 1}$, and $x>0$, let $\mathfrak{S}_n$ be the set of primes $\ell\nmid N$ such that
 \begin{enumerate}
     \item $\ell\equiv 1\mod{p^n}$,
     \item $p\nmid \# \widetilde{E}(\F_\ell)$.
 \end{enumerate}
 For $x>0$, let $\mathfrak{S}_n(x)$ be the set of primes $\ell\in \mathfrak{S}_n$ such that $\ell\leq x$.
 \end{Definition}
 \begin{Lemma}\label{s4 positive density}
 Let $E_{/\Q}$ be an elliptic curve and $p$ a prime such that $p\geq 5$. Assume that the residual representation $\bar{\rho}_{E,p}:\op{Gal}(\bar{\Q}/\Q)\rightarrow \op{GL}_2(\F_p)$ on $E[p]$ is surjective. For $n\in \Z_{\geq 1}$, let $\mathfrak{S}_n$ be as above.
 Then, the set of primes $\mathfrak{S}_n$ has positive density, and furthermore
 \[\lim_{x\rightarrow \infty} \frac{\#\mathfrak{S}_n(x)}{\pi(x)}=\frac{p^2-p-1}{p^{n-1} (p+1)(p-1)^2}.\]
 \end{Lemma}
 
 \begin{proof}
 Given a prime $\ell\nmid Np$, set $a_\ell:=\ell+1-\#\widetilde{E}(\F_\ell)$. Let $\Q(E[p])$ be the extension of $\Q$ which is fixed by the kernel of $\bar{\rho}_{E,p}$. The residual representation gives a natural isomorphism
 \[\op{Gal}(\Q(E[p])/\Q)\xrightarrow{\sim} \op{GL}_2(\F_p).\] Note that the determinant of $\bar{\rho}=\bar{\rho}_{E,p}$ is the mod-$p$ cyclotomic character, hence, \begin{itemize}
     \item $\Q(E[p])$ contains $\Q(\mu_p)$, 
     \item $\bar{\rho}$ induces an isomorphism 
     \[\op{Gal}(\Q(E[p])/\Q(\mu_p))\xrightarrow{\sim} \op{SL}_2(\F_p).\]
 \end{itemize}The prime $\ell$ is unramified in $\Q(E[p])$ and let $\sigma_\ell$ is the Frobenius in $\op{Gal}(\Q(E[p])/\Q)$. The trace and determinant of $\bar{\rho}(\sigma_\ell)$ are as follows,
 \[\op{trace}\left(\bar{\rho}(\sigma_\ell)\right)=a_\ell\text{ and }\op{det}\left(\bar{\rho}(\sigma_\ell)\right)=\ell.\]
 Let $\mathcal{S}'$ be the subset of $\op{Gal}(\Q(E[p])/\Q)$ consisting of $\sigma$ such that
 \[\op{trace}\left(\bar{\rho}(\sigma)\right)\neq 2\text{ and }\op{det}\left(\bar{\rho}(\sigma)\right)=1.\]
 Note that $\sigma_\ell\in \mathcal{S}'$ if and only if $\ell\equiv 1\mod{p}$ and $p\nmid \#\widetilde{E}(\F_\ell)$. On the other hand, C.Jordan showed that $\op{PSL}_2(\F_p)$ is simple, see \cite{jordan}. As a result, $\op{SL}_2(\F_p)$ does not contain any normal subgroup $N$ such that $p$ divides $[\op{SL}_2(\F_p):N]$. Hence the intersection $\Q(E[p])\cap \Q(\mu_{p^n})$ is equal to $\Q(\mu_p)$. Note that for $\sigma\in \mathcal{S}'$, the determinant of $\bar{\rho}(\sigma)$ is equal to $1$, and as a result, $\mathcal{S}'$ is a subset of $\op{Gal}(\Q(E[p])/\Q(\mu_p))$. Let $\Q(E[p], \mu_{p^n})$ be the composite $\Q(E[p])\cdot \Q(\mu_{p^n})$ and note that 
 \[\op{Gal}(\Q(E[p], \mu_{p^n})/\Q(\mu_p))\simeq \op{Gal}\left(\Q(E[p])/\Q(\mu_p)\right)\times \op{Gal}\left(\Q(\mu_{p^n})/\Q(\mu_p)\right).\]
 Let $\mathcal{S}''$ be the subset of $\op{Gal}(\Q(E[p], \mu_{p^n})/\Q(\mu_p))$ such that the projection to $\op{Gal}\left(\Q(E[p])/\Q(\mu_p)\right)$ lies in $\mathcal{S}'$ and the projection to $\op{Gal}\left(\Q(\mu_{p^n})/\Q(\mu_p)\right)$ is trivial. Note that $\sigma_\ell\in \mathcal{S}''$ if and only if 
 \[\ell\equiv 1\mod{p^n}\text{ and } p\nmid \#\widetilde{E}(\F_\ell).\]
 
Note that $\#\mathcal{S}''=\mathcal{S}'$. Thus according to the Chebotarev density theorem, we have that 
 \[\lim_{x\rightarrow \infty} \frac{\#\mathfrak{S}_n(x)}{\pi(x)}=\frac{\#\mathcal{S}'}{p^{n-1}\left(\#\op{GL}_2(\F_p)\right)}>0.\]
 In order to compute $\#\mathcal{S}'$, we need to compute the cardinality of the set of all matrices 
 \[\mtx{a}{b}{c}{d}\in \op{SL}_2(\F_p)\] with trace $2$ and determinant $1$. An elementary matrix calculation shows that \[\#\mathcal{S}'=\# \op{SL}_2(\F_p)-p^2=(p^2-1)p-p^2=p^3-p^2-p.\]
 Putting it all together we have
 \[\lim_{x\rightarrow \infty} \frac{\#\mathfrak{S}_n(x)}{\pi(x)}=\frac{p^2-p-1}{p^{n-1} (p^2-1)(p-1)}.\]
 
 \end{proof}
 
 \begin{proof}[proof of Theorem \ref{thm section 4 main}]
 \par We begin by stating what needs to be proved. We are given a prime $p\geq 11$ contained in $\mathcal{P}_E$. Recall that according to Definition \ref{def PE}, this means that the following conditions hold:
 \begin{enumerate}
     \item $E[p]$ is irreducible as a $\op{Gal}(\bar{\Q}/\Q)$-module,
    \item $E$ has good ordinary reduction at $p$,
    \item $\op{Sel}_{p^\infty}(E/\Q_\infty)=0$.
 \end{enumerate}
 Given $n\geq 1$ and a finite set of primes $\Sigma$, we show that there are infinitely many $\Z/p^n\Z$-extensions $L/\Q$ such that \begin{enumerate}
    \item all primes in $\Sigma$ split in $L$,
     \item $E(L)=E(\Q)$,
     \item $\Sha(E/L)[p^\infty]=\Sha(E/\Q)[p^\infty]=0$.
 \end{enumerate}
 
 Let $t:=\# \Sigma$ and $\mathfrak{S}_n$ be as in Definition \ref{definition sn}. By Lemma \ref{s4 positive density}, the set $\mathfrak{S}_n$ has positive density. All that is required to prove the result is that $\mathfrak{S}_n$ is infinite. Let $\ell_1, \dots, \ell_{t+1}$ be a set of primes in $\mathfrak{S}_n$, and set $m:=\prod_{i=1}^{t+1} \ell_i$. According to Lemma \ref{s4 splitting}, there is a $\Z/p^n\Z$ extension $L/\Q$ such that 
 \begin{enumerate}
     \item all the primes in $\Sigma$ split completely,
     \item $L/\Q$ is ramified outside $\{\ell_i\mid i=1,\dots, t+1\}$.
 \end{enumerate}Given a set of primes $\Omega=\{\ell_1, \dots, \ell_{t+1}\}\subset \mathfrak{S}_n$, let $L_{\Omega}$ be a choice of such a $\Z/p^n\Z$-extension satisfying the above properties. Clearly, there are infinitely many choices of mutually disjoint sets $\Omega$, and hence infinitely fields $L_{\Omega}$. The set of primes that ramify in $L_{\Omega}$ is contained in $\Omega$ which are primes of good reduction of $E$. Said differently, all the primes of $Q_1$ are unramified in $L_{\Omega}$. Furthermore, for each prime $\ell\in \mathfrak{S}_n$, we have arranged that $p\nmid \#\widetilde{E}(\F_\ell)$. Hence, each prime of $Q_2$ is also unramified in $L_{\Omega}$. Hence, it follows from Lemma \ref{stability lemma} that for $L=L_{\Omega}$, both of the following conditions are satisfied
 \begin{enumerate}
     \item $E(L)=E(\Q)$,
     \item $\Sha(E/L)[p^\infty]=\Sha(E/\Q)[p^\infty]=0$.
 \end{enumerate}
 This completes the proof.

 \end{proof}
 
 We conclude this section with a refinement of Theorem \ref{thm section 4 main} for elliptic curves $E_{/\Q}$ for which $E(\Q)_{\op{tors}}\neq 0$.
 \begin{Lemma}\label{boring lemma}
 Let $E_{/\Q}$ be an elliptic curve without CM and assume that
 \begin{enumerate}
     \item $E(\Q)_{\op{tors}}\neq 0$,
     \item $\Sha(E/\Q)$ is finite.
 \end{enumerate} Then $\mathcal{P}_E$ consists of all primes $p$ at which $E$ has good ordinary reduction \emph{except} a finite set of primes $\mathcal{P}_E'$. The set $\mathcal{P}_E'$ consists of a subset of the primes $p$ for which the following conditions are satisfied
 \begin{enumerate}
     \item\label{c1} $p\leq 5$ or $p\mid N$,
     \item $E[p]$ is reducible as a Galois module,
    \item\label{c3} $p\mid \#\Sha(E/\Q)$.
     \item\label{c4} $\# E(\Q)_{\op{tors}}$ is of $p$-power order.
 \end{enumerate}
 \end{Lemma}
\begin{proof}
We refer to the discussion in the proof of Theorem \ref{100 percent}. A prime $p$ at which $E$ has good ordinary reduction is contained in $\mathcal{P}_E$ if and only if $p\nmid \# \Sha(E/\Q)$, $p\nmid c_\ell(E)$ for all primes $\ell\neq p$ and $p\nmid \#\widetilde{E}(\F_p)$. Suppose that $p\notin \mathcal{P}_E$ does not satisfy \eqref{c1}-\eqref{c3}, then, $p$ divides $\#\widetilde{E}(\F_p)$. Since $p\geq 7$, it follows from the Hasse bound that $\#\widetilde{E}(\F_p)=p$. Suppose that there is a prime $\ell\neq p$ that divides $\#E(\Q)_{\op{tors}}$. Denote by $\chi_\ell$ the mod-$\ell$ cyclotomic character.
Since $E(\Q)[\ell]\neq 0$, the mod-$\ell$ representation is of the form \[\bar{\rho}_{E,\ell}=\mtx{1}{\ast}{0}{\chi_{\ell}}.\]
Let $\sigma_p$ be the Frobenius at $p$.
We find that
\[
1+p-\#\widetilde{E}(\F_p)\equiv \op{trace}\left(\bar{\rho}_{E,\ell}(\sigma_p)\right)\equiv 1+\chi_\ell(p)=1+p\pmod{\ell}.
\]
Therefore, $\ell$ divides $\# \widetilde{E}(\F_p)$, a contradiction. Therefore, \eqref{c4} is satisfied, hence the result.
\end{proof}
The following Theorem follows immediately from Theorem \ref{thm section 4 main} and Lemma \ref{boring lemma}. 

\begin{Th}\label{4.15}
Let $E_{/\Q}$ be an elliptic curve without complex multiplication such that $E(\Q)_{\op{tors}}\neq 0$ and $\op{rank} E(\Q)=0$. Then $E$ is diophantine stable \emph{and} $\Sha$-stable at all but a finite set of primes $p$ at which it has good ordinary reduction.
\end{Th}

\par \emph{Example:} Let us illustrate the above result through an explicit example. Consider the elliptic curve $E$ with Cremona label \href{https://www.lmfdb.org/EllipticCurve/Q/14a1/}{14a1}, this is the modular curve $X_0(14)$. Let us note down a few properties of $E$,
\begin{enumerate}
    \item $E$ is a non-CM elliptic curve with rank 0.
    \item At all primes $p\geq 5$ the Galois representation on $E[p]$ is surjective (hence irreducible).
    \item The Tate Shafarevich group $\Sha(E/\Q)=0$.
    \item The group $E(\Q)=\Z/6\Z$, hence, is not of prime power order. 
\end{enumerate}
It follows from Lemma \ref{boring lemma} and Corollary \ref{4.15} that $E$ is diophantine stable at all primes $p\geq 11$ at which $E$ has good ordinary reduction.
\section{Rank jumps and Growth of the Tate-Shafarevich group on average}\label{s5}
\par In the previous section, it was shown that a rational elliptic curve of rank zero without complex multiplication is diophantine stable at $100\%$ of the primes. Furthermore, a similar result was proven for stability of the $p$-primary part of the Tate-Shafarevich group as $p$ varies. In this section, we study the  average growth of ranks and Tate-Shafarevich groups in cyclic $p$-extensions. As in the previous section, results are proved for a fixed elliptic curve $E_{/\Q}$ and varying prime $p$.
\par First, we extend Definition \ref{def PE} to elliptic curves of arbitrary rank. Throughout, it is assumed that the $\Sha(E/\Q)$ is finite and that the $p$-adic regulator $\mathcal{R}_p(E/\Q)$ is non-zero.
\begin{Definition}\label{def 5.1}
For an elliptic curve $E_{/\Q}$, let $\mathcal{P}_E$ consist of the primes $p$ such that
\begin{enumerate}
    \item $p$ is odd,
    \item $E[p]$ is irreducible as a $\op{Gal}(\bar{\Q}/\Q)$-module,
    \item $E$ has good ordinary reduction at $p$,
    \item the equivalent conditions of Lemma \ref{lemma ECF mulambda} are satisfied.
\end{enumerate}
\end{Definition}
Note that when $\op{rank} E(\Q)=0$ and $E(\Q)[p]=0$, the Selmer group $\op{Sel}_{p^\infty}(E/\Q_\infty)$ does not contain any proper finite index submodules, see \cite[Proposition 4.14]{greenbergITEC}. If the $\mu$ and $\lambda$-invariants vanish, then, $\op{Sel}_{p^\infty}(E/\Q_\infty)$ is finite, and hence $0$. Consequently, the above definition coincides with Definition \ref{def PE} when $\op{rank} E(\Q)=0$. It is expected that even in the positive rank setting, the set $\mathcal{P}_E$ consists of $100\%$ of primes, see \cite[Conjecture 1.1]{kunduray1}. Indeed, this is the case if the normalized $p$-adic regulator $\mathcal{R}_p(E/\Q)$ is a $p$-adic unit for $100\%$ of primes $p$, and there is computational evidence for this (see Theorem 3.13 and the subsequent discussion in \emph{loc. cit.}).

\par The main result in this section for elliptic curves of rank $0$ states that for $100\%$ of the primes $p$, the $p$-primary Selmer group exhibits growth for base-change by a large number of $p$-cyclic extensions.
\begin{Th}\label{5.1 main theorem}
Let $E$ be an elliptic curve without complex multiplication such that $\op{rank} E(\Q)=0$ and let $p\in \mathcal{P}_E$. Assume furthermore that the residual representation 
\[\bar{\rho}:\op{Gal}(\bar{\Q}/\Q)\rightarrow\op{GL}_2(\F_p)\] is surjective. Then, for any $n\in \Z_{\geq 1}$ and a finite set of primes $\Sigma$, there are infinitely many $\Z/p^n\Z$-extensions $L/\Q$ such that 
\begin{enumerate}
    \item all primes of $\Sigma$ split in $L$,
    \item $\op{Sel}_{p^\infty}(E/\Q)=0$ and $\op{Sel}_{p^\infty}(E/\Q)\neq 0$.
\end{enumerate}
\end{Th}
\par We obtain an independent result for elliptic curves of positive rank as well. Note that this is the first time in this paper that elliptic curves of positive rank are studied.

\begin{Th}\label{5.1 theorem pos rank}
Let $E$ be an elliptic curve without complex multiplication such that $\op{rank} E(\Q)>0$ and let $p\in \mathcal{P}_E$. Let $L$ be \textit{any} cyclic $p$-extension. Then, at least one of the following assertions hold:
\begin{enumerate}
    \item\label{5.3 c1} $\op{rank}E(L)\geq [L:\Q]\op{rank} E(\Q)$,
    \item\label{5.3 c2} $\mathcal{R}_p(E/\Q)\in\Z_p^\times$ and $\mathcal{R}_p(E/L)\notin \Z_p^\times$,
    \item\label{5.3 c3}
    $\Sha(E/\Q)[p^\infty]=0$ and $\Sha(E/L)[p^\infty]\neq 0$.
\end{enumerate}
\end{Th}
 Thus for any elliptic curve $E_{/\Q}$ without complex multiplicative any triple $(p, n, \Sigma)$, with $p\in \mathcal{P}_E$, there are infinitely many cyclic extensions such that the Selmer group grows. This implies that either the rank increases or the Tate-Shafarevich group grows. The following special case of the conjecture of David-Fearnley-Kisilevsky (see \cite[Conjecture 1.2]{DFK}) predicts that given an elliptic curve over $\Q$, the rank rarely jumps in $\Z/p\Z$-extensions with $p\geq 7$.
 \begin{Conjecture}[David-Fearnley-Kisilevsky]\label{DFK conjecture}
 Let $p\geq 7$ be a prime and $E$ an elliptic curve over $\Q$. Then, there are only finitely many $
\Z/p\Z$-extensions $L/\Q$ such that $\op{rank} E(L)>\op{rank} E(\Q)$.
 \end{Conjecture}
 
 The following is an immediate Corollary to Theorem \ref{5.1 main theorem}. The result shows that if the Conjecture of David-Fearney-Kisilevski is true, then indeed there are many $\Z/p\Z$-extensions in which the size of the Tate-Shafarevich group increases. In fact, infinitely many extensions may be constructed that in which a given finite set of primes split. 
 \begin{Corollary}
 Let $E$ be an elliptic curve without complex multiplication and let $p\in \mathcal{P}_E$. Assume that $p\geq 7$ and that Conjecture \ref{DFK conjecture} is satisfied for $E$ at $p$. Suppose $\op{rank} E(\Q)=0$, then, for any finite set of primes $\Sigma$, there are infinitely many $\Z/p\Z$-extensions $L/\Q$ in which all primes of $\Sigma$ split and
 \[\Sha(E/L)[p^\infty]\neq 0\text{ and }\Sha(E/\Q)[p^\infty]=0.\]
 Suppose that $\op{rank} E(\Q)>0$,
 then, for any finite set of primes $\Sigma$, there are infinitely many $\Z/p\Z$-extensions $L/\Q$ in which all primes of $\Sigma$ split and at least one of the the following conditions is satisfied:
 \begin{enumerate}
     \item $\mathcal{R}_p(E/\Q)\in\Z_p^\times$ and $\mathcal{R}_p(E/L)\in p\Z_p$,
    \item 
    $\Sha(E/\Q)[p^\infty]=0$ and $\Sha(E/L)[p^\infty]\neq 0$.
 \end{enumerate}
 \end{Corollary}
 \begin{proof}
\par First, consider the case when $\op{rank}E(\Q)=0$. According to Theorem \ref{5.1 main theorem}, there are an infinitude of $\Z/p\Z$-extensions $L/\Q$ in which the primes of $\Sigma$ split and such that $\op{Sel}_{p^\infty}(E/L)\neq 0$ and $\op{Sel}_{p^\infty}(E/\Q)=0$. If Conjecture \ref{DFK conjecture} is true, then for all but finitely many of the extensions $L/\Q$, the rank of $E(L)$ is zero. Suppose $L/\Q$ is a $\Z/p\Z$-extension such that $\op{rank} E(L)=0$. Note that since $p\in \mathcal{P}_E$, it is assumed that $E[p]$ is irreducible as a Galois module. As a consequence, we have that $E(\Q)[p]=0$. Since $L/\Q$ is a $p$-extension, it follows that $E(L)[p]=0$ as well, see Lemma \ref{NSW lemma}. As a result, the $p$-primary Mordell-Weil group $E(L)\otimes \Q_p/\Z_p=0$, and hence, the Selmer group $\op{Sel}_{p^\infty}(E/L)$ coincides with $\Sha(E/L)[p^\infty]$. Since $\op{Sel}_{p^\infty}(E/L)\neq 0$, the result follows.
\par When $\op{rank} E(\Q)>0$, the result is a direct consequence of Theorem \ref{5.1 theorem pos rank}. \end{proof}
 \begin{Lemma}\label{boringlemma}
Let $p\geq 5$ be a prime and $L/\Q$ a $p$-cyclic extension, and $\ell\neq p$ be a prime. Make the following assumptions:
\begin{enumerate}
    \item $c_\ell^{(p)}(E/\Q)=1$,
    \item if $\ell$ is ramified in $L$, then $E$ has good reduction at $\ell$.
\end{enumerate}Then, we have that $\prod_{v\mid\ell} c_v^{(p)}(E/L)=1$, where $v$ ranges over all primes of $L$ above $\ell$.
\end{Lemma}
\begin{proof}
Since $\ell\neq p$, and $L/\Q$ is a pro-$p$, it follows that $\ell$ is tamely ramified in $L$, and the results for base-change of Tamagawa numbers in \cite[Table 1, pp. 556-557]{Kida} apply. We consider two cases.
\begin{enumerate}
    \item First consider the case where $\ell$ is unramified in $L$. Fix a prime $v\mid \ell$. Since $p\geq 5$, the Tamagawa number $c_v^{(p)}(E/L)\neq 1$ if and only if the Kodaira type of $E_{/L_v}$ is $\op{I}_n$ for an integer $n\in \Z_{\geq 1}$ that is divisible by $p$ (see \cite[p. 448]{Sil09}). However, according to \cite[Table 1, pp. 556-557]{Kida}, the only way this is possible is if if the Kodaira type of $E_{/\Q_\ell}$ is $\op{I}_n$ to begin with. However, this is not the case, since $c_\ell^{(p)}(E/\Q)= 1$.
    \item Next, consider the case when $\ell$ ramifies in $L$. In this case, it is assumed that $E$ has good reduction at $\ell$, and the result is clear.
\end{enumerate}
\end{proof}
 \begin{Proposition}\label{boring prop}
 Let $E_{/\Q}$ be an elliptic curve and $p$ an odd prime at which $E$ has good ordinary reduction. Assume that $p\in \mathcal{P}_E$ and that $\op{rank} E(\Q)=0$.
 Let $L/\Q$ be a $\Z/p^n\Z$-extension of $\Q$ for which the following conditions are satisfied:
 \begin{enumerate}
     \item\label{boring prop c1} $E$ has good reduction at all primes $\ell\neq p$ which ramify in $L$,
     \item\label{boring prop c2} there exists a prime $\ell\neq p$ which ramifies in $L$, we have that $p|\#\widetilde{E}(\F_\ell)$.
 \end{enumerate}
 Then at least one of the following holds
 \begin{enumerate}
     \item $\op{rank} E(L)>0$,
     \item $\Sha(E/L)[p^\infty]\neq 0$ and $\Sha(E/\Q)[p^\infty]=0$.
 \end{enumerate}
 \end{Proposition}
 \begin{proof}
 Recall that since $\op{rank} E(\Q)=0$, the (normalized) regulator $\mathcal{R}_p(E/\Q)=1$. By definition, since $p\in \mathcal{P}_E$, the following equivalent conditions are satisfied:
 \begin{enumerate}
     \item $\mu_p(E/\Q)=0$ and $\lambda_p(E/\Q)=0$, 
     \item\label{29 aug c2} $\chi(\Gamma_\Q, E[p^\infty])=1$, 
     \item\label{29 aug c3} $\Sha(E/\Q)[p^\infty]=0$, $p\nmid c_\ell(E/\Q)=1$ for all $\ell\neq p$ and $p\nmid \#\widetilde{E}(\F_p)$.
 \end{enumerate}The equivalence of \eqref{29 aug c2} and \eqref{29 aug c3} follows from the Euler characteristic formula. Assume by way of contradiction that both the conditions \eqref{boring prop c1} and \eqref{boring prop c2} of the proposition are not satisfied. In other words, assume that $\op{rank} E(L)=0$ and $\Sha(E/L)[p^\infty]=0$. Note that since $E[p]$ is irreducible as a Galois module, it follows that $E(\Q)[p]=0$. Since $L/\Q$ is a $p$-cyclic extension, it follows from Lemma \ref{NSW lemma} that $E(L)[p]=0$. Thus, we have that 
\[\chi_t(\Gamma_L, E[p^\infty])\sim \# \Sha(E/L)[p^\infty] \times \prod_{v\nmid p} c_{v}^{(p)}(E/L) \times \prod_{v|p}\left(\# \widetilde{E}(\kappa_v)[p^\infty]\right)^2.\] 
We show that $\chi_t(\Gamma_L, E[p^\infty])$ by showing that the contributing factors are all $1$. It follows from Lemma \ref{boringlemma} that the Tamagawa product $\prod_{v\nmid p} c_v^{(p)}(E/L)=1$. Finally, for each prime $v|p$, $[k_v:\F_p]$ is either $1$ or a power of $p$. Hence, according to
\[\#\widetilde{E}(\F_p)[p]=0\Rightarrow \#\widetilde{E}(k_v)[p]=0.\] Putting everything together, we conclude that $\chi_t(\Gamma_L, E[p^\infty])=1$. Hence, according to Lemma \ref{lemma ECF mulambda}, this implies that $\mu_p(E/L)=0$ and $\lambda_p(E/L)=0$. On the other hand, since there is a prime $\ell\neq p$ which is ramified in $L$ such that $p|\#\widetilde{E}(\F_\ell)$, we have that $\ell\in Q_2$. Therefore, it follows from Kida's formula that $\lambda_p(E/L)>0$, a contradiction.
\end{proof}
\begin{proof}[proof of Theorem \ref{5.1 main theorem}]
Note that since $p\in \mathcal{P}_E$, the Selmer group $\op{Sel}_{p^\infty}(E/\Q_\infty)=0$. As argued before, the irreducibility of $E[p]$ implies that the Selmer group $\op{Sel}_{p^\infty}(E/\Q)$ injects into $\op{Sel}_{p^\infty}(E/\Q_\infty)$, hence is $0$ as well. Let $\mathfrak{T}_n$ be the set of primes $\ell\nmid Np$ such that 
 \begin{enumerate}
     \item $\ell\equiv 1\mod{p^n}$,
     \item $p\mid \# \widetilde{E}(\F_\ell)$.
 \end{enumerate}
 The same arguments as in the proof of Lemma \ref{s4 positive density} show that $\mathfrak{T}_n$ has positive density, hence, is infinite. Let $t:=\#\Sigma$, and let $\Omega$ be a set of $t+1$ primes in $\mathfrak{T}_n$. There is a $\Z/p^n\Z$-extension $L_\Omega$ with ramification contained in $\Omega$ and in which $\Sigma$ is split. The extension $L_\Omega$ is ramified at at least one prime $\ell\in \Omega$. Since $p$ divides $\#\widetilde{E}(\F_\ell)$, the result follows from Proposition \ref{boring prop}.
\end{proof}

\begin{proof}[proof of Theorem \ref{5.1 theorem pos rank}]
\par The proof of this result is similar to that of Proposition \ref{boring prop}. Assume that the conditions \eqref{5.3 c2} and \eqref{5.3 c3} are not satisfied. Then, we show that \eqref{5.3 c1} is satisfied, i.e., $\op{rank}E(L)\geq p\op{rank} E(\Q)$. Recall that since $p\in \mathcal{P}_E$, $\chi_t(\Gamma_\Q, E[p^\infty])=1$. Since $E[p]$ is irreducible as a Galois module, the Euler characteristic formula implies that $\mathcal{R}_p(E/\Q)\in \Z_p^\times$. Since \eqref{5.3 c2} is not satisfied, $\mathcal{R}_p(E/L)\in \Z_p^\times$. Moreover $\Sha(E/\Q)[p^\infty]=0$, and since \eqref{5.3 c3} is not satisfied, it follows that $\Sha(E/L)[p^\infty]=0$. By the same reasoning as in the proof of Proposition \ref{boring prop}, we have that $\chi_t(\Gamma_L, E[p^\infty])=1$. Therefore, by Lemma \ref{lemma ECF mulambda}, $\lambda_p(E/L)=\op{rank} E(L)$. On the other hand, since $p\in \mathcal{P}_E$, $\lambda_p(E/\Q)=\op{rank} E(\Q)$. Therefore, according to \eqref{kidaformula}, 
\[\op{rank} E(L)\geq [L:\Q] \op{rank} E(\Q),\] hence, condition \eqref{5.3 c1} is satisfied.
 \end{proof}

 \section{Stability for Elliptic curves on average at a fixed prime $p$}
 \par Recall that in sections \ref{s4} and \ref{s5}, results were proved for a fixed elliptic curve and varying prime $p$. In this section, we study the dual questions. 
 \begin{question}\label{question 6.1}
 Let $p$ be a prime. What can one say about the proportion of rational elliptic curves that are diophantine stable at $p$?
 \end{question}
 Our results on diophantine stability only apply to $p\geq 11$, hence, we make this assumption throughout. Recall that any elliptic curve $E_{/\Q}$ admits a unique Weierstrass equation
\begin{equation}\label{weier}
E:Y^2 = X^3 + aX + b
\end{equation}
where $a, b$ are integers and $\gcd(a^3 , b^2)$ is not divisible by any twelfth power.
Since $p\geq 5$, such an equation is minimal.
Recall that the \emph{height of} $E$ satisfying the minimal equation $\eqref{weier}$ is given by $H(E) := \op{max}\left(|a|^3, b^2\right)$.
Let $\mathscr{E}$ be the set of isomorphism classes of elliptic curves defined over $\Q$, and for any subset $\mathcal{S}\subset \mathscr{E}$, let $\mathcal{S}(x)$ consist of all $E\in \mathcal{S}$ such that $H(E)<x$.
\par In order to provide some answers to the above question, we assume some well known conjectures.
 \begin{Conjecture}[Rank distribution conjecture]
 For elliptic curves ordered by height, the proportion of elliptic curves with rank $0$ is $50\%$, the proportion with rank $1$ is $50\%$.
 \end{Conjecture}
 Thus, the proportion of elliptic curves $E_{/\Q}$ with $\op{rank} E(\Q)\geq 2$ is expected to be $0\%$. Denote by $\mathscr{E}$ the set of all isomorphism classes of rational elliptic curves. Let $\mathscr{E}'$ be the set of isomorphism classes of rational elliptic curves of rank 0 with good reduction at the primes $2$ and $3$. At any prime $\ell$, the proportion of elliptic curves with good reduction is $(1-\ell^{-1})$, see the proof of \cite[Proposition 4.2 (1)]{stats2}. It is reasonable to make the following conjecture.
 \begin{Conjecture}[RDC+]\label{RDC}
The lower density of $\mathscr{E}'$ satisfies
\[\liminf_{x\rightarrow \infty} \frac{\# \mathscr{E}'(x)}{\# \mathscr{E}(x)}= \frac{1}{2}\left(1-\frac{1}{2}\right)\left(1-\frac{1}{3}\right)=\frac{1}{6}.\]
 \end{Conjecture}
Along with the above rank distribution conjecture, one also makes the following assumption due to Delaunay, see \cite{delaunay}. \begin{Conjecture}[Del-p]\label{Del}
The proportion of all elliptic curves such that $E$ has rank 0 for which $p\mid \# \Sha(E/\Q)$ is given by
\[\frac{1}{2}\left(1-\prod_{i\geq 1} \left(1-\frac{1}{p^{2i-1}}\right)\right).\]
\end{Conjecture}
Delaunay predicts that the proportion of rank 0 elliptic curves with $p$ dividing the order of the Tate-Shafarevich group is given by \[1-\prod_{i\geq 1} \left(1-\frac{1}{p^{2i-1}}\right)= \frac{1}{p}+\frac{1}{p^3}-\frac{1}{p^4}+\dots.\] The term $1/2$ is expected due to rank distribution.
\par For $(a,b)\in \F_p\times \F_p$ with $\Delta:=4a^3+27b^2$ non-zero, associate the elliptic curve $E_{a,b}$ defined by the Weierstrass equation 
\[
E_{a,b}:Y^2=X^3+aX+b.
\]
Let $\mathfrak{T}_p$ be the set of pairs $(a,b)\in \F_p\times \F_p$ such that $\#E_{a,b}(\F_p)\in \{p,p+1\}$. For $t\in \{0,1\}$, the number of isomorphism classes of elliptic curves over $\F_p$ with $\#E(\F_p)=p+1-t$ is given by $H(t^2-4p)$, where, $H(\cdot)$ denotes the Kronecker class number (see \cite[p. 184]{Schoof87}).

\begin{Lemma}\label{lemma 6.5}
Assume that $p\geq 5$ is a prime. There is an explicit positive constant $c_1$ such that
\[
\# \mathfrak{T}_p  \leq c_1 p^{\frac{3}{2}} \log p \left( \log \log p\right)^2.
\]
\end{Lemma}
\begin{proof}
It is well known that $E_{a,b}$ is isomorphic to $E_{a',b'}$ over $\F_p$ if and only if 
\[a'=c^4 a\text{ and }b'=c^6 b\] for some element $c\in \F_p^\times$. Thus, the number of curves $E_{a',b'}$ that are isomorphic to $E_{a,b}$ is at most $\frac{p-1}{2}$, hence, the first inequality. The inequality follows from \cite[Proposition 1.9]{Lenstra_annals}.
\end{proof}

\par The following result provides a conditional answer to Question \ref{question 6.1}.
\begin{Th}\label{last thm}
Let $p\geq 11$ be a prime and assume that Conjectures (RDC+) and (Del-p) are satisfied. Let $\mathscr{E}_p$ be the set of isomorphism classes of elliptic curves $E_{/\Q}$ satisfying the following conditions:
\begin{enumerate}
    \item\label{6.5 c1} $\op{rank} E(\Q)=0$,
    \item $E[p]$ is irreducible as a Galois module,
    \item $E$ has good reduction at $2,3$,
    \item\label{6.5 c4} $E$ has good ordinary reduction at $p$,
    \item\label{6.5 c5} $E$ is diophantine stable and $\Sha$-stable at $p$.
\end{enumerate}
Then, for the effective constant $c_1>0$ of Lemma \ref{lemma 6.5}, we have that
\[\begin{split}\liminf_{x\rightarrow \infty} \frac{\# \mathscr{E}_p(x)}{\#\mathscr{E}(x)} \geq &\frac{1}{6}-\frac{1}{p}-\frac{1}{2}\left(1-\prod_i \left(1-\frac{1}{p^{2i-1}}\right)\right) \\
& -(\zeta(p)-1)-\zeta(10) c_1\frac{\log p\cdot (\log \log p)^2}{\sqrt{p}}.\\
\end{split}\]
In particular, we find that
\[\liminf_{p\rightarrow \infty}\left(\liminf_{x\rightarrow \infty} \frac{\# \mathscr{E}_p(x)}{\#\mathscr{E}(x)}\right)=\frac{1}{6}.\]
\end{Th}
\begin{proof}
Suppose that the conditions \eqref{6.5 c1}-\eqref{6.5 c4} are satisfied for $(E,p)$. Then, if $p\in \mathcal{P}_E$ it follows from Theorem \ref{thm section 4 main} that $E$ is diophantine stable and $\Sha$-stable at $p$.
For $i=1,\dots, 5$ let $\mathscr{E}_p^{(i)}$ be the subset of $\mathscr{E}$ defined as follows:
\begin{enumerate}
    \item $\mathscr{E}_p^{(1)}$ is the subset of all elliptic curves of rank 0 such that $p|\#\Sha(E/\Q)$. According to (Del-p),
    \[\limsup_{x\rightarrow \infty} \frac{\#\mathscr{E}_p^{(1)}(x)}{\#\mathscr{E}(x)}=\frac{1}{2}\left(1-\prod_i \left(1-\frac{1}{p^{2i-1}}\right)\right).\]
    \item Let $\mathscr{E}_p^{(2)}$ be the subset of all elliptic curves such that $p\mid \prod_{\ell\neq p} c_\ell(E/\Q)$. It follows from \cite[Corollary 8.8]{HKR} that 
    \[\limsup_{x\rightarrow \infty} \frac{\#\mathscr{E}_p^{(2)}(x)}{\#\mathscr{E}(x)}\leq (\zeta(p)-1) \]
    \item Let $\mathscr{E}_p^{(3)}$ be the subset of all elliptic curves such that at least of the following conditions hold:
    \begin{enumerate}
        \item $p\mid \#\widetilde{E}(\F_p)$,
        \item $p$ has good supersingular reduction at $p$. 
    \end{enumerate} Note that since $p\geq 11$, it follows from the Hasse bound that this is the case precisely when $\#\widetilde{E}(\F_p)=p+1-t$, where $t\in \{0,1\}$. We have that
    \[\limsup_{x\rightarrow \infty} \frac{\#\mathscr{E}_p^{(3)}(x)}{\#\mathscr{E}(x)}\leq \zeta(10)\frac{\#\mathfrak{T}_p}{p^2}, \]see the proof of \cite[Theorem 4.14]{kunduray1}. It follows from Lemma \ref{lemma 6.5} that 
    \[\#\mathfrak{T}_p\leq c_1 p^{\frac{3}{2}} \log p \cdot (\log \log p)^2.\]
    Therefore, we have that 
    \[\limsup_{x\rightarrow \infty} \frac{\#\mathscr{E}_p^{(3)}(x)}{\#\mathscr{E}(x)}\leq \frac{\zeta(10) c_1\log p\cdot (\log \log p)^2}{\sqrt{p}}. \]
    \item Let $\mathscr{E}_p^{(4)}$ be the subset of all elliptic curves such that $E[p]$ is reducible as a Galois module. It follows from the main result of \cite{duke} that  \[\limsup_{x\rightarrow \infty} \frac{\#\mathscr{E}_p^{(4)}(x)}{\#\mathscr{E}(x)}=0. \]
     \item Let $\mathscr{E}_p^{(5)}$ be the subset of all elliptic curves such $E$ has bad reduction at $p$, it follows from \cite[Proposition 4.2 (1)]{stats2} that \[\lim_{x\rightarrow \infty} \frac{\#\mathscr{E}_p^{(5)}(x)}{\#\mathscr{E}(x)}= \frac{1}{p}. \]
\end{enumerate}
Note that $\mathscr{E}_p$ is contained in $\mathscr{E}'\backslash \left(\bigcup_{i=1}^5 \mathscr{E}_p^{(i)}\right)$. Therefore,
\[\liminf_{x\rightarrow \infty} \frac{\# \mathscr{E}_p(x)}{\#\mathscr{E}(x)}\geq \liminf_{x\rightarrow \infty} \frac{\# \mathscr{E}'(x)}{\#\mathscr{E}(x)}-\sum_{i=1}^5 \limsup_{x\rightarrow \infty} \frac{\# \mathscr{E}_p^{(i)}(x)}{\#\mathscr{E}(x)}.\]
The result follows.
\end{proof}


\begin{thebibliography}{1}

\bibitem{BKR}Lea Beneish, Debanjana Kundu, and Anwesh Ray. "Rank Jumps and Growth of Shafarevich--Tate Groups for Elliptic Curves in $\mathbb {Z}/p\mathbb {Z} $-Extensions." arXiv preprint arXiv:2107.09166 (2021).

\bibitem{csslinks}Coates, John, Peter Schneider, and Ramdorai Sujatha. "Links between cyclotomic and $\op{GL}_2$ Iwasawa theory." (2003).


\bibitem{DFK}
Chantal David, Jack Fearnley, and Hershy Kisilevsky.
\newblock Vanishing of ${L}$-functions of elliptic curves over number fields.
\newblock {\em Ranks of elliptic curves and random matrix theory}, (341):247,
  2007.
  
\bibitem{delaunay}Delaunay, Christophe. "Heuristics on Tate-Shafarevitch groups of elliptic curves defined over Q." Experimental Mathematics 10.2 (2001): 191-196.

\bibitem{Dok}Dokchitser, Tim. "Ranks of elliptic curves in cubic extensions." arXiv preprint math/0510533 (2005).

\bibitem{duke}Duke, William. "Elliptic curves with no exceptional primes." Comptes rendus de l'Académie des sciences. Série 1, Mathématique 325.8 (1997): 813-818.

\bibitem{GJN}González–Jiménez, Enrique, and Filip Najman. "Growth of torsion groups of elliptic curves upon base change." Mathematics of Computation 89.323 (2020): 1457-1485.

\bibitem{greenbergITEC}Greenberg, Ralph. "Iwasawa theory for elliptic curves." Arithmetic theory of elliptic curves. Springer, Berlin, Heidelberg, 1999. 51-144.

\bibitem{HKR}Hatley, Jeffrey, Debanjana Kundu, and Anwesh Ray. "Statistics for Anticyclotomic Iwasawa Invariants of Elliptic Curves." arXiv preprint arXiv:2106.01517 (2021).

\bibitem{hachimorimatsuno}Hachimori, Yoshitaka, and Kazuo Matsuno. "An analogue of Kida's formula for the Selmer groups of elliptic curves." Journal of Algebraic Geometry 8.3 (1999): 581-601.

\bibitem{jordan}Jordan, Camille. Traité des substitutions et des équations algébriques. Vol. 1. Gauthier-Villars, 1870.

\bibitem{katozeta}Kato, Kazuya. "p-adic Hodge theory and values of zeta functions of modular forms." Astérisque 295 (2004): 117-290.

\bibitem{Kida}
Masanari Kida.
\newblock Variation of the reduction type of elliptic curves under small base
  change with wild ramification.
\newblock {\em Central European J. Math.}, 1(4):510--560, 2003.


\bibitem{kida}Kida, Y\^uji. "l-extensions of CM-fields and cyclotomic invariants." Journal of Number Theory 12.4 (1980): 519-528.

\bibitem{kobayashi}Kobayashi, Shin-ichi. "Iwasawa theory for elliptic curves at supersingular primes." Inventiones mathematicae 152.1 (2003): 1-36.

\bibitem{kunduray1}Kundu, Debanjana, and Anwesh Ray. "Statistics for Iwasawa invariants of elliptic curves." accepted for publication in the Transactions of the AMS (2021).

\bibitem{Lenstra_annals}Lenstra Jr, Hendrik W. "Factoring integers with elliptic curves." Annals of mathematics (1987): 649-673.

\bibitem{stats2}Kundu, Debanjana, and Anwesh Ray. "Statistics for Iwasawa Invariants of elliptic curves, $\rm {II} $." arXiv preprint arXiv:2106.12095 (2021).

\bibitem{mazurrationalpoints}Mazur, Barry. "Rational points of abelian varieties with values in towers of number fields." Inventiones mathematicae 18.3-4 (1972): 183-266.

\bibitem{mazurrubin}Mazur, Barry, Karl Rubin, and Michael Larsen. "Diophantine stability." American Journal of Mathematics 140.3 (2018): 571-616.

\bibitem{murty}Murty, V. Kumar. "Modular forms and the Chebotarev density theorem II." London Mathematical Society Lecture Note Series (1997): 287-308.

\bibitem{NSW}Neukirch, J\"urgen, Alexander Schmidt, and Kay Wingberg. "Cohomology of number fields." Vol. 323. Springer Science $\&$ Business Media, 2013.

\bibitem{PR82}Perrin-Riou, Bernadette. "Descente infinie et hauteur $p$-adique sur les courbes elliptiques {\`a} multiplication complexe." Inventiones mathematicae 70.3 (1982): 369-398.

\bibitem{ray1}Ray, Anwesh, and Ramdorai Sujatha. "Euler characteristics and their congruences in the positive rank setting." Canadian Mathematical Bulletin 64.1 (2021): 228-245.

\bibitem{raymulti}Ray, Anwesh, and R. Sujatha. "Euler Characteristics and their Congruences for Multi-signed Selmer Groups." arXiv preprint arXiv:2011.05387 (2020).

\bibitem{raymuinvariant}Ray, Anwesh, and R. Sujatha. "On the $\mu $-invariants of residually reducible Galois representations." arXiv preprint arXiv:2104.10797 (2021).

\bibitem{schneiderhp1}Schneider, Peter. "p-adic height pairings I." Inventiones mathematicae 69.3 (1982): 401-409.

\bibitem{Schneider85}Schneider, Peter. "p-adic height pairings. II." Inventiones mathematicae 79.2 (1985): 329-374.

\bibitem{Schoof87}Schoof, Ren\'e. "Nonsingular plane cubic curves over finite fields." Journal of combinatorial theory, Series A 46.2 (1987): 183-211.

\bibitem{serreopenimage}Serre, Jean-Pierre. "Propri\'et\'es Galoisiennes des points d’ordre fini des courbes elliptiques." Invent. math 15 (1972): 259-331.

\bibitem{serrechebotarev}Serre, Jean-Pierre. "Quelques applications du théoreme de densité de Chebotarev." Publications Mathématiques de l'Institut des Hautes Études Scientifiques 54.1 (1981): 123-201.


\bibitem{Sil09}
Joseph~H Silverman.
\newblock {\em The arithmetic of elliptic curves}, volume 106 of {\em Graduate
  Texts in Mathematics}.
\newblock Springer, 2009.

\bibitem{washington1997}Washington, Lawrence C. Introduction to cyclotomic fields. Vol. 83. Springer Science $\&$ Business Media, 1997.

\bibitem{zerbes}Zerbes, Sarah Livia. "Generalised Euler characteristics of Selmer groups." Proceedings of the London Mathematical Society 98.3 (2009): 775-796.

\end{thebibliography}
\end{document}